\documentclass[10pt,twocolumn,twoside]{IEEEtran}
\usepackage[english]{babel }

\usepackage[margin=54pt]{geometry}

\title{\vspace*{18pt}  Distributed Controllers for Multi-Terminal HVDC~Transmission~Systems}

\author{ \IEEEauthorblockA{Martin Andreasson$^{\IEEEauthorrefmark{1}}$, Dimos V. Dimarogonas, Henrik Sandberg and  Karl H. Johansson  
}
\\ 
\IEEEauthorblockA{ACCESS Linnaeus Centre, KTH Royal Institute of Technology, Stockholm, Sweden. }
\thanks{This work was supported in part by the European Commission by  the Swedish Research Council and the Knut and Alice Wallenberg Foundation. 
The $2^{\text{nd}}$ author is also affiliated with the Centre for Autonomous Systems at KTH. This work extends the manuscripts presented in \cite{Andreasson2014_IFAC, andreasson2014CDC}. The Authors are with the ACCESS Linnaeus Centre, KTH Royal Institute of Technology, 11428 Stockholm, Sweden.
\IEEEauthorrefmark{1} Corresponding author. E-mail: mandreas@kth.se}
}
\IEEEoverridecommandlockouts

\usepackage{mathrsfs}
\usepackage{amsfonts}
\usepackage{amssymb}
\usepackage{amsmath}
\usepackage{graphicx}
\usepackage{caption}
\usepackage{subcaption}
\usepackage{amsthm}
\usepackage{commath}
\usepackage{tikz}
\usepackage{tkz-graph}
\usepackage[europeanresistors]{circuitikz}
\usetikzlibrary{arrows}
\usepackage{flushend}
\usepackage{psfrag}
\usepackage{pgfplots}
\usepackage{todonotes}
\usetikzlibrary{external}
\usepackage{booktabs}
\usepackage{blockdiagram}

\newtheorem{theorem}{Theorem}
\newtheorem{corollary}[theorem]{Corollary}
\newtheorem{lemma}[theorem]{Lemma}
\newtheorem{remark}{Remark}
\newtheorem{objective}{Objective}

\DeclareMathOperator*{\argmin}{argmin}

\DeclareMathOperator*{\diag}{diag}

\DeclareMathOperator*{\sgn}{sgn}

\newcommand{\beq}{\begin{equation}}
\newcommand{\eeq}{\end{equation}}
\newcommand{\bq}{\begin{eqnarray}}
\newcommand{\eq}{\end{eqnarray}}
\newcommand{\bqn}{\begin{eqnarray*}}
\newcommand{\eqn}{\end{eqnarray*}}
\newcommand{\bee}{\begin{enumerate}}
\newcommand{\eee}{\end{enumerate}}

\newlength\fheight
\newlength\fwidth

\begin{document}
\maketitle

\begingroup
\makeatletter
\renewcommand{\p@subfigure}{}
\makeatother

\begin{abstract}
High-voltage direct current (HVDC) is an increasingly commonly used technology for long-distance electric power transmission, mainly due to its low resistive losses. 
In this paper the voltage-droop method (VDM) is reviewed, and three novel distributed controllers for multi-terminal HVDC (MTDC) transmission systems are proposed. Sufficient conditions for when the proposed controllers render the equilibrium of the closed-loop system asymptotically stable are provided. These conditions give insight into suitable controller architecture, e.g., that the communication graph should be identical with the graph of the MTDC system, including edge weights.
Provided that the equilibria of the closed-loop systems are asymptotically stable, it is shown that the voltages asymptotically converge to within predefined bounds. Furthermore, a quadratic cost of the injected currents is asymptotically minimized. The proposed controllers are evaluated on a four-bus MTDC system. 
\end{abstract}

\section{Introduction}
The transmission of power over long distances is one of the greatest challenges in today's power transmission systems. Since resistive losses increase with the length of power transmission lines, higher voltages have become abundant in long-distance power transmission. One example of long-distance power transmission are large-scale off-shore wind farms, which often require power to be transmitted in cables over long distances to the mainland AC power grid. High-voltage direct current (HVDC) power transmission is a commonly used technology for long-distance power transmission. Its higher investment costs compared to AC transmission lines, mainly due to expensive AC-DC converters, are compensated by its lower resistive losses for sufficiently long distances \cite{Bresesti2007}. The break-even point, i.e., the point where the total costs of overhead HVDC and AC lines are equal, is typically 500-800 km \cite{padiyar1990hvdc}. However, for cables, the break-even point is typically lower than 100 km, due to the AC current needed to charge the capacitors of the cable insulation \cite{bresesti2007hvdc}. Increased use of HVDC for electrical power transmission suggests that future HVDC transmission systems are likely to consist of multiple terminals connected by several HVDC transmission lines. Such systems are referred to as multi-terminal HVDC (MTDC) systems in the literature \cite{van2010multi}. 

Maintaining an adequate DC voltage is an important  control problem for HVDC transmission systems. Firstly, the voltage levels at the DC terminals govern the current flows by Ohm's law and Kirchhoff's circuit laws. Secondly, if the DC voltage deviates too far from a nominal operational voltage, equipment may be damaged, resulting in loss of power transmission capability \cite{van2010multi}.
For existing point-to-point HVDC connections consisting of only two buses, the voltage is typically controlled at one of the buses, while the injected current is controlled at the other bus \cite{kundur1994power}. 
As this decentralized controller structure has no natural extension to the case with three or more buses, various methods have been proposed for controlling MTDC systems. The voltage margin method, VMM, is an extension of the controller structure of point-to-point connections. For an $n$-bus MTDC system, $n-1$ buses are assigned to control the injected current levels around a setpoint, while the remaining bus controls the voltage around a given setpoint. VMM typically controls the voltage fast and accurately. The major disadvantage is the undesirable operation points, which can arise when one bus alone has to change its current injections to maintain a constant voltage level. While this can be addressed by assigning more than one bus to control the voltage, it often leads to undesirable switching of the injected currents \cite{Nakajima1999}. 
The voltage droop method, VDM, on the other hand is symmetric, in the sense that all local decentralized controllers have the same structure. Each bus injects current in proportion to the local deviation of the bus voltage from its nominal value. Similar to VMM, VDM is a simple decentralized controller not relying on any communication \cite{haileselassie2009control, johnson1993expandable}. As we will formally show in this paper, a major disadvantage of VDM is
 static errors of the voltage, as well as possibly suboptimal operation points.

The highlighted drawbacks of existing decentralized MTDC controllers give rise to the question if performance can be increased by allowing for communication between buses. 
Distributed controllers have been successfully applied to both primary, secondary, and to some extent also tertiary frequency control of AC transmission systems \cite{andreasson2013distributed, simpson2012synchronization, li2014connecting, Andreasson2014TAC, mallada2014optimal}. Although the dynamics of HVDC grids can be modelled with a lower order model than AC grids, controlling DC grids may prove more challenging. This is especially true for decentralized and distributed controller structures. The challenges consist of the faster time-scales of MTDC systems, as well as the lack of a globally measurable variable corresponding to the AC frequency. 
In \cite{dai2010impact, dai2011voltage, silva2012provision, andreasson2014mtdacac}, decentralized controllers are employed to share primary frequency control reserves of AC systems connected through an MTDC system. Due to the lack of a communication network, the controllers induce static control errors. 
In \cite{Sarlette20123128}, a distributed approach is taken in contrast to the previous references, allowing for communication between DC buses and thus improving the performance of the controller. In \cite{Anand2013, Xiaonan2014}, distributed voltage controllers for DC microgrids achieving current sharing are proposed. The controllers however relies on a complete communication network. In \cite{Nasirian2014}, a distributed controller for DC microgrids with an arbitrary, connected communication network is proposed. Stability of the closed-loop system is however not guaranteed.  
 
In this paper three novel distributed controllers for MTDC transmission systems are proposed, all allowing for certain limited communication between buses. It is shown that under certain conditions, the proposed controllers render the equilibrium of the closed-loop system asymptotically stable. Additionally the voltages converge close to their nominal values, while a quadratic cost function of the current injections is asymptotically minimized. The sufficient stability criteria derived in this paper give insights into suitable controller architecture, as well as insight into the controller design. All proposed controllers are evaluated by simulation on a four-bus MTDC system.

The remainder of this paper is organized as follows. In Section~\ref{sec:prel}, the mathematical notation is defined. In Section \ref{sec:model}, the system model and the control objectives are defined. 
In Section~\ref{sec:general_results}, some generic properties of MTDC systems are derived.  
In Section~\ref{subsec:model_droop}, voltage droop control is analyzed. 
 In Section~\ref{sec:dist_control}, three different distributed averaging controllers are presented, and their stability and steady-state properties are analyzed. In Section~\ref{sec:simulations}, simulations of the distributed controllers on a four-terminal MTDC system are provided, showing the effectiveness of the proposed controllers. The paper ends with a concluding discussion in Section~\ref{sec:discussion}.

\section{Notation}
\label{sec:prel}
Let $\mathcal{G}$ be a graph. Denote by $\mathcal{V}=\{ 1,\hdots, n \}$ the vertex set of $\mathcal{G}$, and by $\mathcal{E}=\{ 1,\hdots, m \}$ the edge set of $\mathcal{G}$. Let $\mathcal{N}_i$ be the set of neighboring vertices to $i \in \mathcal{V}$.
In this paper we will only consider static, undirected and connected graphs. For the application of control of MTDC power transmission systems, this is a reasonable assumption as long as there are no power line failures.
Denote by $\mathcal{B}$ the vertex-edge incidence matrix of a graph, and let $\mathcal{\mathcal{L}_W}=\mathcal{B}W\mathcal{B}^T$ be its weighted Laplacian matrix, with edge-weights given by the  elements of the positive definite diagonal matrix $W$.
Let $\mathbb{C}^-$ denote the open left half complex plane, and $\bar{\mathbb{C}}^-$ its closure. We denote by $c_{n\times m}$ a matrix of dimension $n\times m$ whose elements are all equal to $c$, and by $c_n$ a column vector whose elements are all equal to $c$. For a symmetric matrix $A$, $A>0 \;(A\ge 0)$ is used to denote that $A$ is positive (semi) definite. $I_{n}$ denotes the identity matrix of dimension $n$. For vectors $x$ and $y$, we denote by $x\le y$ that the inequality holds for all elements. 
We will often drop the notion of time dependence of variables, i.e., $x(t)$ will be denoted $x$ for simplicity.

\section{Model and problem setup}
\label{sec:model}
Consider an MTDC transmission system consisting of $n$ HVDC terminals, henceforth referred to as buses. The buses are denoted by the vertex set $ \mathcal{V}= \{1, \dots, n\}$, see Figure~\ref{fig:graph} for an example of an MTDC topology. The DC buses are modelled as ideal current sources which are connected by $m$ HVDC transmission lines, denoted by the edge set $ \mathcal{E}= \{1, \dots, m\}$. The dynamics of any system (e.g., an AC transmission system) connected through the DC buses are neglected, as are the dynamics of the DC buses (e.g., AC-DC converters). The HVDC lines are assumed to be purely resistive, neglecting capacitive and inductive elements of the HVDC lines. The assumption of purely resistive lines is not restrictive for the control applications considered in this paper  \cite{kundur1994power}. This implies that 
\begin{align*}
I_{ij} = \frac{1}{R_{ij}} (V_i -V_j),
\end{align*}
due to Ohm's law, where $V_i$ is the voltage of bus $i$, $R_{ij}$ is the resistance and $I_{ij}$ is the current of the HVDC line from bus $i$ to $j$. 
The voltage dynamics of an arbitrary DC bus $i$ are thus given by
\begin{align}
C_i \dot{V}_i &= -\sum_{j\in \mathcal{N}_i} I_{ij} + I_i^{\text{inj}} + u_i \nonumber \\
&= -\sum_{j\in \mathcal{N}_i} \frac{1}{R_{ij}}(V_i -V_j) + I_i^{\text{inj}} + u_i,
\label{eq:hvdc_scalar}
\end{align}
where $C_i$ is the total capacitance of bus $i$, including shunt capacitances and the capacitance of the HVDC line, $I_i^{\text{inj}}$ is the injected current due to loads, which is assumed to be unknown and constant (assuming step disturbances), and $u_i$ is the controlled injected current. Note that we impose no dynamics nor constraints on the controlled injected current $u_i$. In practice, this requires that each MTDC bus is connected with a strong AC grid which can supply sufficient power to the MTDC grid. 
Equation \eqref{eq:hvdc_scalar} may be written in vector-form as
\begin{align}
\begin{aligned}
C\dot{V} &= -\mathcal{L}_R V +I^{\text{inj}} + u,
\end{aligned}
\label{eq:hvdc_vector}
\end{align}
where $V=[V_1, \dots, V_n]^T$, $C=\diag([C_1, \dots, C_n])$, $I^{\text{inj}} = [I^{\text{inj}}_1, \dots, I^{\text{inj}}_n]^T$, $u=[u_1, \dots, u_n]^T$ and $\mathcal{L}_R$ is the weighted Laplacian matrix of the graph representing the transmission lines, whose edge-weights are given by the conductances $\frac{1}{R_{ij}}$. For convenience, we also introduce the matrix of elastances $E=\diag([C_1^{-1}, \dots, C_n^{-1}])$.  The control objective considered in this paper is defined below.

\begin{figure}
\center
\begin{tikzpicture}[american voltages]
\draw 
(0,2) node[ground, rotate=180] {}
to 	[C, l=$C_1$] (0,0.5)
to 	[short, -*]  (0,0)
to  [R, l_=$R_{12}$, i_=$I_{12}$] (3,0)
to  [short, -*] (4,0)
to	(4,0.5) [C, l=$C_2$]
to	(4,2) node[ground, rotate=180] {}
(0.4,0.42) node[]  {$V_1$}
(3.6,0.42) node[]  {$V_2$}
[short, l={}]  (0,0)
to  [R, l=$R_{13}$, i=$I_{13}$] (0,-2.5)
to  [short, -*] (0,-3)
[short, l={}]  (4,0)
to  [R, l_=$R_{24}$, i_=$I_{24}$] (4,-2.5)
to  [short, -*] (4,-3)
(0,-5) node[ground] {}
to 	[C, l_=$C_3$] (0,-3.5)
to 	[short, l={}]  (0,-3)
to  [R, l=$R_{34}$, i=$I_{34}$] (3,-3)
to  [short, l={}] (4,-3)
to	(4,-3.5) [C, l_=$C_4$]
to	(4,-5) node[ground] {}
(0.4,-3.42) node[]  {$V_3$}
(3.6,-3.42) node[]  {$V_4$}
(-2,0) node[ground, rotate=180] {}
to [american current source, l^=$I_1^\text{inj}+u_1$] (0,0)
to [short, l={}] (0,0)
(6,0) node[ground, rotate=180] {}
to [american current source, l_=$I_2^\text{inj}+u_2$] (4,0)
to [short, l={}] (4,0)
(-2,-3) node[ground] {}
to [american current source, l_=$I_3^\text{inj}+u_3$] (0,-3)
to [short, l={}] (0,-3)
(6,-3) node[ground] {}
to [american current source, l^=$I_4^\text{inj}+u_4$] (4,-3)
to [short, l={}] (4,-3);
\end{tikzpicture}
\caption{Topology of a four bus MTDC system.}
\label{fig:graph}
\end{figure}
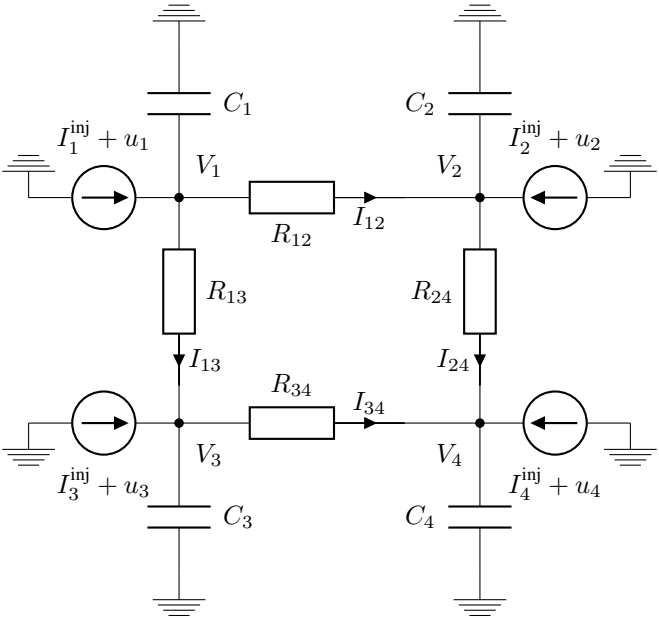

\begin{objective}
\label{obj:1}
The cost of the current injections should be minimized asymptotically. More precisely
\begin{align}
\lim_{t\rightarrow \infty} u(t) = u^*,
\label{eq:opt_u}
\end{align}
where $u^*$ minimizes the cost of current injections while ensuring a balanced network, and is defined by
\begin{align}
[u^*,V^*]=\argmin_{[u,V]} \sum_{i \in \mathcal{V}} \frac{1}{2} f_i u_i^2 \quad \text{s.t.} \quad \mathcal{L}_R V  &= I^{\text{inj}} + u, \label{eq:opt1}
\end{align}
and where $f_i>0, i=1, \dots , n$, are any positive constants. Subsequently, the following quadratic cost function of the voltage deviations should be minimized over the set $V^*$:
\begin{align}
\label{eq:cost_voltage}
\min_{V\in V^*} \sum_{i \in \mathcal{V}} \frac 12 g_i(V^\text{nom}_i -V_i)^2,
\end{align}
for some $g_i \ge 0 \; \forall i=1, \dots, n$, 
and where $V^{\text{nom}}_i$ is the nominal voltage of bus $i$ and $V^*$ is obtained by solving \eqref{eq:opt1}. 
\end{objective}

\begin{remark}
Equations \eqref{eq:opt_u}--\eqref{eq:opt1} imply that  the asymptotic voltage differences between the DC buses are bounded, i.e., $\lim_{t\rightarrow \infty} |V_i(t)-V_j(t)|\le \Delta V \; \forall i,j \in \mathcal{V}$, for some $\Delta V>0$. 
This implies that it is in general not possible to have $\lim_{t\rightarrow \infty}V_i(t) = V^{\text{nom}}_i$ for all $i \in \mathcal{V}$, e.g., by PI-control. 
 We show in Lemma~\ref{lemma:voltage_differences} that $\Delta V$ can be bounded by a function of the injected and controlled injected currents, $I^\text{inj} +u$, as well as the Laplacian matrix of the MTDC system. 
\end{remark}

\begin{remark}
\label{cor:optimality_general}
The optimal solution $V^*$ of \eqref{eq:opt_u}--\eqref{eq:opt1} is unique only up to an additive constant vector $c_n$, where all elements are equal. Minimizing \eqref{eq:cost_voltage} determines this constant vector, which can be seen as the average voltage in the MTDC grid. 
\end{remark}

\begin{remark}
Equations~\eqref{eq:opt_u}--\eqref{eq:opt1} are analogous to the quadratic optimization of AC power generation costs, c.f., \cite{andreasson2013distributed, dorfler2014breaking}. The quadratic cost function of the voltages \eqref{eq:cost_voltage} has no analogy in the corresponding secondary AC frequency control problem. This since the voltages in an MTDC grid do not synchronize in general, as opposed to the frequencies in an AC grid. 
\end{remark}

\begin{remark}
The quadratic cost of voltage deviations \eqref{eq:cost_voltage} replaces the common notion of acceptable voltage range. 
\end{remark}

\section{General properties of MTDC systems}
\label{sec:general_results}  
Before exploring different control strategies for MTDC systems, we derive some general results on properties of controlled MTDC systems which will be useful for the remainder of this paper. Our first result gives a generic upper bound on the asymptotic relative voltage differences of an MTDC system, regardless of the controller structure. 

 \begin{lemma}
 \label{lemma:voltage_differences}
Consider any stationary control signal $u$. The relative voltage differences satisfy 
\begin{align*}
|V_i-V_j| \le 2 I^{\text{max}} \sum_{i=2}^n \frac{1}{\lambda_i},
\end{align*}
where $I^{\text{max}}=\max_i |I_i^{\text{tot}}|$ and $I^{\text{tot}}_i=  I_i^{\text{inj}} +u_i$ and $\lambda_i$ denotes the $i$'th eigenvalue of $\mathcal{L}_R$. 
\end{lemma}
\begin{proof}
Consider the equilibrium of  \eqref{eq:hvdc_vector}:
\begin{align}
\label{eq:equilibrium_I_tot}
\mathcal{L}_R V &=  I^{\text{inj}} + u \triangleq I^{\text{tot}}.
\end{align}
Let
$
V=\sum_{i=1}^n a_i w_i,
$
where $w_i$ is the $i$'th eigenvector of $\mathcal{L}_R$ with the corresponding eigenvalue $\lambda_i$. Since $\mathcal{L}_R$ is symmetric, the eigenvectors $\{w_i\}_{i=1}^n$ can be chosen so that they form an orthonormal basis of $\mathbb{R}^n$. Using the eigendecomposition of $V$ above, we obtain the following equation from \eqref{eq:equilibrium_I_tot}:
\begin{align}
\label{eq:equilibrium_I_tot_eigen}
\mathcal{L}_R V &= \mathcal{L}_R \sum_{i=1}^n a_i w_i =  \sum_{i=1}^n a_i \lambda_i w_i = I^{\text{tot}}.
\end{align}
By premultiplying \eqref{eq:equilibrium_I_tot_eigen} with $w_k$ for $k=1, \dots , n$, we obtain:
\begin{align*}
a_k\lambda_k = w_k^T I^{\text{tot}},
\end{align*}
due to orthonormality of $\{w_i\}_{i=1}^n$. Hence, for $k=2, \dots, n$ we get
\begin{align*}
a_k = \frac{w_k^T I^{\text{tot}}}{\lambda_k}.
\end{align*}
The constant $a_1$ is not determined by \eqref{eq:equilibrium_I_tot_eigen}, since $\lambda_1=0$. Denote $\Delta V=\sum_{i=2}^n a_i w_i$. Since $w_1=\frac{1}{\sqrt{n}}1_{n}$, $V_i-V_j=\Delta V_i-\Delta V_j$ for any $i,j \in \mathcal{V}$. Thus, the following bound is easily obtained:
\begin{align*}
&|V_i-V_j| = |\Delta V_i-\Delta V_j| \le 2 \max_i |\Delta V_i| = 2 \norm{\Delta V}_\infty \\
 &\le 2 \norm{\Delta V}_2 = 2 \norm{\sum_{i=2}^n a_i w_i}_2 \le 2 \sum_{i=2}^n |a_i|= 2 \sum_{i=2}^n \left|\frac{w_i^T I^{\text{tot}}}{\lambda_i}\right| \\
 & \le 2 I^{\text{max}} \sum_{i=2}^n \frac{1}{\lambda_i},
\end{align*}
where we have used the fact that $\norm{w_i}_2=1$ for all $i=1, \dots, n$, and $\norm{x}_\infty \le \norm{x}_2$ for any $x\in \mathbb{R}^n$.
\end{proof}

Our second result reveals an interesting general structure of asymptotically optimal MTDC control signals.

 \begin{lemma}
 \label{lemma:optimality}
Equations~\eqref{eq:opt_u}--\eqref{eq:opt1} in Objective \ref{obj:1} are satisfied if and only if $\lim_{t\rightarrow \infty} u(t) =\mu F^{-1}1_{n}$ and $\lim_{t\rightarrow \infty} \mathcal{L}_R V(t)  = I^{\text{inj}} + \mu F^{-1}1_{n}$, where $F = \diag([f_1, \dots, f_n])$. The scaling factor is given by $\mu = -(\sum_{i=1}^n I_i^{\text{inj}})/(\sum_{i=1}^n f_i^{-1})$. 
 \end{lemma}
 
 \begin{proof}
The KKT condition for the optimization problem \eqref{eq:opt1} is $Fu=\mu 1_{n}$, which gives $u= F^{-1}\mu 1_{n}$. Substituting this expression for $u$ and pre-multiplying the constraint $\lim_{t\rightarrow \infty} \mathcal{L}_R V(t)  = I^{\text{inj}} + F^{-1} \mu 1_{n}$ with $1_n^T$, yields the desired expression for $\mu$. 
Since \eqref{eq:opt1} is convex, the KKT condition is a necessary and sufficient condition for optimality.
 \end{proof}
 
 \begin{lemma}
\label{cor:optimality_voltage}
Equation~\eqref{eq:cost_voltage}  in Objective \ref{obj:1} is minimized if and only if $\sum_{i =1}^n g_i( V_i  - V^\text{nom}_i ) = 0$. 
\end{lemma}
\begin{proof}
By considering the equilibrium of \eqref{eq:hvdc_vector}, the relative voltages $\Delta V$ are uniquely determined by $I^\text{inj}$ and $u$. Thus $V=\Delta V + k 1_n$, for some $k\in \mathbb{R}$. Taking the derivative of the quadratic cost function \eqref{eq:cost_voltage} with respect to $k$ thus corresponds to the  necessary and sufficient KKT condition for optimality, and yields:
\begin{align*}
\frac{\partial}{\partial k}\sum_{i \in \mathcal{V}} \frac 12 g_i(V^\text{nom}_i -V_i)^2  = \sum_{i =1}^n g_i( V_i  - V^\text{nom}_i ) = 0.
\end{align*} 
\end{proof}
 
\begin{remark}
The choice of the controller gains as detailed in Lemma~\ref{lemma:optimality}, is analogous to the controller gains in the AC frequency controller being inverse proportional to the coefficients of the quadratic generation cost function \cite{dorfler2014breaking}. 
\label{rem:optimality_AC}
\end{remark}
  
\section{Voltage droop control}
\label{subsec:model_droop}
In this section the VDM will be studied, as well as some of its limitations. VDM is a simple decentralized proportional controller taking the form
\begin{align}
u_i &= K^P_i(V_i^{\text{nom}}-V_i),
\label{eq:droop}
\tag{VDM}
\end{align}
where $V^{\text{nom}}$ is the nominal DC voltage. Alternatively, the controller \eqref{eq:droop} can be written in vector form as
\begin{align}
u &= K^P(V^{\text{nom}}-V),
\label{eq:droop_vector}
\end{align}
where $V^{\text{nom}}=[V^{\text{nom}}_1, \dots, V^{\text{nom}}_n]^T$ and $K^P=\diag([K^P_1, \dots, K^P_n])$.
The decentralized structure of the voltage droop controller is often advantageous for control of HVDC buses, as the time constant of the voltage dynamics is typically smaller than the communication delays between the buses. The DC voltage regulation is typically carried out by all buses. However, VDM possesses some severe drawbacks. Firstly, the voltages of the buses don't converge to a value close to the nominal voltages in general. 
Secondly, the controlled injected currents do not converge to the optimal value. 

\begin{theorem}
\label{th:droop_stability}
Consider an MTDC network described by \eqref{eq:hvdc_scalar}, where the control input $u_i$ is given by \eqref{eq:droop} and the injected currents $I_i^{\text{inj}}$ are constant. The equilibrium of the closed-loop system is stable for any $K^P>0$, in the sense that the voltages $V$ converge to some constant value. In general, Objective~\ref{obj:1} is not satisfied. However, the controlled injected currents satisfy $\lim_{t\rightarrow \infty} \sum_{i=1}^n (u^i+I_i^{\text{inj}}) = 0$. 
\end{theorem}

\begin{proof}
The closed-loop dynamics of \eqref{eq:hvdc_vector} with $u$ given by \eqref{eq:droop} are
\begin{align}
\begin{aligned}
\dot{V} &= -E\mathcal{L}_R V + EK^P(V^{\text{nom}}-V) + EI^{\text{inj}} \\
&= \underbrace{-E(\mathcal{L}_R + K^P)}_{\triangleq A}V + EK^PV^{\text{nom}} + EI^{\text{inj}}.
\end{aligned}
\label{eq:hvdc_closed_loop_droop_vector}
\end{align}
Clearly the equilibrium of \eqref{eq:hvdc_closed_loop_droop_vector} is stable if and only if $A$ as defined above is Hurwitz. Consider the characteristic polynomial of $A$:
\begin{align*}
\begin{aligned}
0&= \det(sI_n-A) = \det\left(sI_n + E(\mathcal{L}_R + K^P)\right) \\  \Leftrightarrow
0&= \det \underbrace{\left(sC + (\mathcal{L}_R + K^P)\right)}_{\triangleq Q(s)}.
\end{aligned}
\end{align*}
The equation $0=\det\left(Q(s)\right)$ has a solution for a given $s$ only if $0=x^TQ(s)x$ has a solution for some $\norm{x}_2=1$. This gives
\begin{align*}
0&=s\underbrace{x^TCx}_{a_1} + \underbrace{x^T(\mathcal{L}_R + K^P)x}_{a_0}.
\end{align*}
Clearly $a_0, a_1>0$, which implies that the above equation has all its solutions $s\in \mathbb{C}^-$ by the Routh-Hurwitz stability criterion. This implies that the solutions of $0=\det(Q(s))$ satisfy $s\in \mathbb{C}^-$, and thus that $A$ is Hurwitz.

Now consider the equilibrium of \eqref{eq:hvdc_closed_loop_droop_vector}:
\begin{align}
0&= {-(\mathcal{L}_R + K^P)}V + K^PV^{\text{nom}} + I^{\text{inj}}.
\label{eq:hvdc_closed_loop_droop_vector_equilibrium}
\end{align}
Since $K^P>0$ by assumption $(\mathcal{L}_R + K^P)$ is invertible, which implies
\begin{align}
V=(\mathcal{L}_R + K^P)^{-1}\left( K^PV^{\text{nom}} + I^{\text{inj}} \right),
\label{eq:hvdc_closed_loop_droop_vector_equilibrium_voltage}
\end{align}
which does not satisfy Objective \ref{obj:1}, in general. By inserting \eqref{eq:hvdc_closed_loop_droop_vector_equilibrium_voltage} in \eqref{eq:droop_vector}, it is easily seen that
$$ u \ne \mu F^{-1} 1_{n}$$ in general. By Lemma \ref{lemma:optimality}, Objective \ref{obj:1} is thus in general not satisfied. Premultiplying \eqref{eq:hvdc_closed_loop_droop_vector_equilibrium} with $1_n^T C^{-1}$ yields
\begin{align*}
0 &= 1_n^T K^P(V^{\text{nom}} -V) + I^{\text{inj}} = \sum_{i=1}^n (u_i + I_i^{\text{inj}}). \qedhere
\end{align*} 
\end{proof}
Next, we construct explicitly a class of droop-controlled MTDC systems for which Objective~\ref{obj:1} is never satisfied. 
\begin{lemma}
\label{lemma:voltage_droop_obj_1_example}
Consider an MTDC network described by \eqref{eq:hvdc_scalar}, where the control input $u_i$ is given by \eqref{eq:droop} and the injected currents $I_i^{\text{inj}}\ne 0_n$ satisfy either $I_i^{\text{inj}} \le 0_n$ or $I_i^{\text{inj}} \ge 0_n$, where the inequality is strict for at least one element. Furthermore let $V^\text{nom}= v^\text{nom}1_n$. Then Equation~\eqref{eq:cost_voltage} in Objective~\ref{obj:1} is not minimized, regardless of the system and controller parameters.  
\end{lemma}
\begin{proof}
The equilibrium of the closed-loop dynamics is given by
\begin{align}
\label{eq:eq_droop_counterexample}
(\mathcal{L}_R+K^P)(V-v^\text{nom}1_n) = I^\text{inj}.
\end{align}
For convenience, define $\bar{V}=V-v^\text{nom}1_n$. 
Without loss of generality, assume that $I_i^{\text{inj}} \le 0_n$. By premultiplying \eqref{eq:eq_droop_counterexample} with $1_n^T$, we obtain $1_n^TK^P\bar{V} < 0$. This implies that for at least one index $i_1$, $\bar{V}_{i_1} < 0$. Assume for the sake of contradiction that there exists an index $i_2$ such that $\bar{V}_{i_2} \ge 0$. We can without loss of generality assume $\bar{V}_{i_2} \ge \bar{V}_i \; \forall i \ne i_2$. By considering the $i_2$th element of \eqref{eq:eq_droop_counterexample}, we obtain
\begin{align*}
K^P_{i_2} \bar{V}{i_2} + \sum_{j \in \mathcal{N}_{i_2}} (\bar{V}_{i_2} - \bar{V}_j) \le 0.
\end{align*}
This implies that for at least one $j \in \mathcal{N}_{i_2}$ we have $\bar{V}_j > \bar{V}_{i_2}$, contradicting the assumption that $\bar{V}_{i_2} \ge \bar{V}_i \; \forall i \ne i_2$. Thus, $\bar{V} < 0$, and \eqref{eq:cost_voltage} in Objective~\ref{obj:1} can clearly not be minimized. 
\end{proof}
Generally when tuning the proportional gains $K^P$, there is a trade-off between the voltage errors and the optimality of the current injections. Low gains $K^P$ will result in closer to optimal current injections, but the voltages will be far from the reference value. On the other hand, having high gains  $K^P$ will ensure that the voltages converge close to the nominal voltage, at the expense of large deviations from the optimal current injections $u^*$. This rule of thumb is formalized in the following theorem.

\begin{theorem}
\label{th:droop_power_sharing}
Consider an MTDC network described by \eqref{eq:hvdc_scalar}, where the control input $u_i$ is given by \eqref{eq:droop} with positive gains $K^P_i=f_i^{-1}$, and constant injected currents $I_i^{\text{inj}}$. The DC voltages satisfy
\begin{IEEEeqnarray*}{llCl}
\lim_{K^P \rightarrow \infty } &\lim_{t\rightarrow \infty} V(t) &=& V^{\text{nom}} \\
\lim_{K^P \rightarrow 0} &\lim_{t\rightarrow \infty} V(t) &=& \sgn\big(\sum_{i=1}^n I_i^{\text{inj}}\big)\infty 1_n ,
\end{IEEEeqnarray*}
while the controlled injected currents satisfy
\begin{IEEEeqnarray*}{llCl}
\lim_{K^P \rightarrow \infty } &\lim_{t\rightarrow \infty} u(t) &=& -I^{\text{inj}} \\
\lim_{K^P \rightarrow 0} &\lim_{t\rightarrow \infty} u(t) &=& u^*,
\end{IEEEeqnarray*}
where the notation means
\begin{IEEEeqnarray*}{rCCCCCl}
K^P \rightarrow & \infty &\Leftrightarrow & K_i^P &\rightarrow & \infty \; & \forall i=1, \dots, n \\
K^P \rightarrow & 0 &\Leftrightarrow & K_i^P&\rightarrow & 0 \; & \forall i=1, \dots, n.
\end{IEEEeqnarray*}
\end{theorem}

\begin{proof}
Let us first consider the case when $K^P \rightarrow \infty$. In the equilibrium of \eqref{eq:hvdc_closed_loop_droop_vector}, the voltages satisfy by \eqref{eq:hvdc_closed_loop_droop_vector_equilibrium_voltage}:
\begin{align*}
 \lim_{K^P \rightarrow \infty}  V &= \lim_{K^P \rightarrow \infty} (\mathcal{L}_R + K^P)^{-1}\left( K^PV^{\text{nom}} + I^{\text{inj}} \right) \\
&= \lim_{K^P \rightarrow \infty} ( K^P)^{-1}\left( K^PV^{\text{nom}} + I^{\text{inj}} \right) = V^{\text{nom}}.
\end{align*}
By inserting the above expression for the voltages, the controlled injected currents are given by
\begin{align*}
\lim_{K^P \rightarrow \infty} u &= \lim_{K^P \rightarrow \infty} K^P\left( V^{\text{nom}} -V   \right) \\
&= \lim_{K^P \rightarrow \infty} K^P \left( - (K^P)^{-1}I^{\text{inj}}  \right){=}{-}I^{\text{inj}}.
\end{align*}

Now consider the case when ${K^P \rightarrow 0}$. Since $(\mathcal{L}_R + K^P)$ is real and symmetric, any vector in $\mathbb{R}^n$ can be expressed as a linear combination of its eigenvectors. Denote by $(v_i, \lambda_i)$ the eigenvector and eigenvalue pair $i$ of $(\mathcal{L}_R + K^P)$. Write
\begin{align}
\left( K^PV^{\text{nom}} + I^{\text{inj}} \right) &= \sum_{i=1}^n a_i v_i,
\label{eq:eigendecomposition_1}
\end{align}
where $a_i, i=1, \dots, n$ are real constants. The equilibrium of \eqref{eq:hvdc_closed_loop_droop_vector} implies that the voltages satisfy
\begin{align*}
\lim_{K^P \rightarrow 0}  V &= \lim_{K^P \rightarrow 0} (\mathcal{L}_R + K^P)^{-1}\left( K^PV^{\text{nom}} + I^{\text{inj}} \right) \\
&= \lim_{K^P \rightarrow 0} (\mathcal{L}_R + K^P)^{-1}\sum_{i=1}^n a_i v_i \\&=\lim_{K^P \rightarrow 0} \sum_{i=1}^n \frac{a_i}{\lambda_i} v_i = \frac{a_1}{\lambda_1} v_1,
\end{align*}
where $\lambda_1$ is the smallest eigenvalue of $(\mathcal{L}_R + K^P)$, which clearly satisfies $\lambda_1 \rightarrow 0^+$ as ${K_i^P\rightarrow 0 \; \forall i=1, \dots, n}$. Hence the last equality in the above equation holds. By letting ${K^P \rightarrow 0}$ and premultiplying   \eqref{eq:eigendecomposition_1} with $v_1^T=1/n1_{n}$, we obtain $a_1=(\frac{1}{n} \sum_{i=1}^n I_i^{\text{inj}})$ since the eigenvectors of $(\mathcal{L}_R + K^P)$ form an orthonormal basis of $\mathbb{R}^n$. Thus $\lim_{K^P \rightarrow 0} \lim_{t\rightarrow \infty} V(t) = \sgn\left(\sum_{i=1}^n I_i^{\text{inj}}\right)\infty_n$. 
Finally the controlled injected currents are given by
\begin{align*}
\lim_{K^P \rightarrow 0} u &= \lim_{K^P \rightarrow 0} K^P(V^{\text{nom}} - V) \\
&= \lim_{K^P \rightarrow 0} K^P\big(V^{\text{nom}} - \frac{a_1}{\lambda_1} 1_{n}\big) = -\frac{a_1}{\lambda_1} K^P 1_{n}.
\end{align*}
By premultiplying \eqref{eq:hvdc_closed_loop_droop_vector_equilibrium} with $1_n^T C^{-1}$ we obtain
\begin{align*}
1_n^T K^P(V^{\text{nom}} - V) = - 1_n^T I^{\text{inj}},
\end{align*}
which implies that
\begin{align*}
\frac{a_1}{\lambda_1} &= \frac{1_n^T I^{\text{inj}}}{1_n^T  K^P 1_{n} } = (\sum_{i=1}^n I_i^{\text{inj}})/(\sum_{i=1}^n K^P_i),
\end{align*}
which gives $u=u^*$ due to Lemma \ref{lemma:optimality}. 
\end{proof}

\section{Distributed MTDC control}
\label{sec:dist_control}
The shortcomings of the VDM control, as indicated in Theorem \ref{th:droop_stability}, motivate the development of novel controllers for MTDC networks. In this section we present three distributed controllers for MTDC networks, allowing for communication between HVDC buses. The use of a  communication network allows for distributed controllers, all fulfilling Objective~\ref{obj:1} but with specific advantages and disadvantages. 
Controllers \eqref{eq:distributed_voltage_control} and \eqref{eq:distributed_voltage_control_complete} have the advantage of only requiring $n$ additional controller variables, but \eqref{eq:distributed_voltage_control} suffers from poor redundancy and \eqref{eq:distributed_voltage_control_complete} requires a complete communication topology. The controller \eqref{eq:distributed_voltage_control_third} does not suffer from poor redundancy and can be implemented using any connected distributed  communication topology, but at the cost of requiring $2n$ additional controller variables. 
 The architectures of the controllers proposed later on in this section are illustrated in Figure~\ref{fig:control_architecture}.

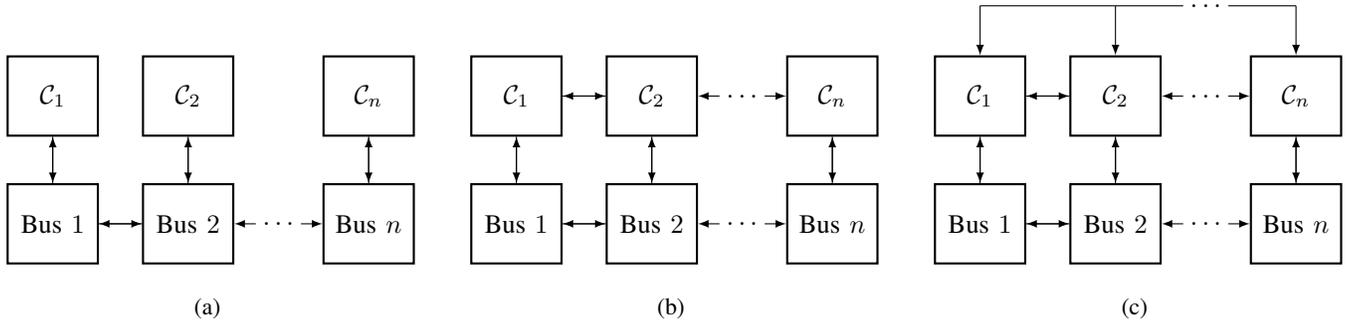
\begin{figure*}
\begin{center}

\begin{subfigure}[b]{0.3\textwidth}
\begin{center}
\begin{tikzpicture}[auto, >=latex] 
	
	\node [block, node distance =1.0cm, minimum width=1.2cm] (b1) {Bus $1$};
	\node [block, right of=b1, node distance =1.8cm, minimum width=1.2cm] (b2) {Bus $2$};
	\node [block, right of=b2, node distance =2.4cm, minimum width=1.2cm] (bn) {Bus $n$};
	\node [block, above of=b1,node distance =1.7cm, minimum width=1.2cm] (c1) {$\mathcal{C}_1$};
	\node [block, above of=b2,node distance =1.7cm, minimum width=1.2cm] (c2) {$\mathcal{C}_2$};
	\node [block, above of=bn,node distance =1.7cm, minimum width=1.2cm] (cn) {$\mathcal{C}_n$};
	
	\node [output, right of=b2,node distance =0.9cm] (b2r) {};
	\node [output, left  of=bn,node distance =0.9cm] (bnl) {};
	\node [draw=none,fill=none, right of=b2,node distance =1.2cm] (b2ndots) {$\ldots$};
		
	
	\draw [connector] (b1) -- (b2) {};
	\draw [connector] (b2) -- (b1) {};
	\draw [connector] (b2r) -- (b2) {};
	\draw [connector] (bnl) -- (bn) {};

	\draw [connector] (b1) -- (c1) {};
	\draw [connector] (c1) -- (b1) {};
	\draw [connector] (b2) -- (c2) {};
	\draw [connector] (c2) -- (b2) {};	
	\draw [connector] (bn) -- (cn) {};
	\draw [connector] (cn) -- (bn) {};	
	
\end{tikzpicture}
\end{center}
\caption{}
\label{fig:C.a}
\end{subfigure}
\qquad 
\begin{subfigure}[b]{0.3\textwidth}
\begin{center}
\begin{tikzpicture}[auto, >=latex] 
	
	\node [block, node distance =1.0cm, minimum width=1.2cm] (b1) {Bus $1$};
	\node [block, right of=b1, node distance =1.8cm, minimum width=1.2cm] (b2) {Bus $2$};
	\node [block, right of=b2, node distance =2.4cm, minimum width=1.2cm] (bn) {Bus $n$};
	\node [block, above of=b1,node distance =1.7cm, minimum width=1.2cm] (c1) {$\mathcal{C}_1$};
	\node [block, above of=b2,node distance =1.7cm, minimum width=1.2cm] (c2) {$\mathcal{C}_2$};
	\node [block, above of=bn,node distance =1.7cm, minimum width=1.2cm] (cn) {$\mathcal{C}_n$};
	
	\node [output, right of=b2,node distance =0.9cm] (b2r) {};
	\node [output, left  of=bn,node distance =0.9cm] (bnl) {};
	\node [draw=none,fill=none, right of=b2,node distance =1.2cm] (b2ndots) {$\ldots$};
	\node [output, right of=c2,node distance =0.9cm] (c2r) {};
	\node [output, left  of=cn,node distance =0.9cm] (cnl) {};
	\node [draw=none,fill=none, right of=c2,node distance =1.2cm] (c2ndots) {$\ldots$};

	
	\draw [connector] (b1) -- (b2) {};
	\draw [connector] (b2) -- (b1) {};
	\draw [connector] (b2r) -- (b2) {};
	\draw [connector] (bnl) -- (bn) {};
	
	\draw [connector] (c1) -- (c2) {};
	\draw [connector] (c2) -- (c1) {};
	\draw [connector] (c2r) -- (c2) {};
	\draw [connector] (cnl) -- (cn) {};
	
	\draw [connector] (b1) -- (c1) {};
	\draw [connector] (c1) -- (b1) {};
	\draw [connector] (b2) -- (c2) {};
	\draw [connector] (c2) -- (b2) {};	
	\draw [connector] (bn) -- (cn) {};
	\draw [connector] (cn) -- (bn) {};	
	
\end{tikzpicture}
\end{center}
\caption{}
\label{fig:C.b}
\end{subfigure}
\qquad 
\begin{subfigure}[b]{0.3\textwidth}
\begin{center}
\begin{tikzpicture}[auto, >=latex]
	
	\node [block, node distance =1.0cm, minimum width=1.2cm] (b1) {Bus $1$};
	\node [block, right of=b1, node distance =1.8cm, minimum width=1.2cm] (b2) {Bus $2$};
	\node [block, right of=b2, node distance =2.4cm, minimum width=1.2cm] (bn) {Bus $n$};
	\node [block, above of=b1,node distance =1.7cm, minimum width=1.2cm] (c1) {$\mathcal{C}_1$};
	\node [block, above of=b2,node distance =1.7cm, minimum width=1.2cm] (c2) {$\mathcal{C}_2$};
	\node [block, above of=bn,node distance =1.7cm, minimum width=1.2cm] (cn) {$\mathcal{C}_n$};


	\node [output, above of=c1,node distance =1.2cm] (h1) {};
	\node [output, above of=c2,node distance =1.2cm] (h2) {};
	\node [output, above of=cn,node distance =1.2cm] (hn) {};
	
	\node [output, right of=b2,node distance =0.9cm] (b2r) {};
	\node [output, left  of=bn,node distance =0.9cm] (bnl) {};
	\node [draw=none,fill=none, right of=b2,node distance =1.2cm] (b2ndots) {$\ldots$};
	
	\node [output, right of=c2,node distance =0.9cm] (c2r) {};
	\node [output, left  of=cn,node distance =0.9cm] (cnl) {};
	\node [draw=none,fill=none, right of=c2,node distance =1.2cm] (c2ndots) {$\ldots$};
	
	\node [output, right of=h2,node distance =0.9cm] (h2r) {};
	\node [output, left  of=hn,node distance =0.9cm] (hnl) {};
	\node [draw=none,fill=none, right of=h2,node distance =1.2cm] (h2ndots) {$\ldots$};
	
	
	\draw [connector] (b1) -- (b2) {};
	\draw [connector] (b2) -- (b1) {};
	\draw [connector] (b2r) -- (b2) {};
	\draw [connector] (bnl) -- (bn) {};
	
	\draw [connector] (c1) -- (c2) {};
	\draw [connector] (c2) -- (c1) {};
	\draw [connector] (c2r) -- (c2) {};
	\draw [connector] (cnl) -- (cn) {};
	
	\draw [connector] (b1) -- (c1) {};
	\draw [connector] (c1) -- (b1) {};
	\draw [connector] (b2) -- (c2) {};
	\draw [connector] (c2) -- (b2) {};	
	\draw [connector] (bn) -- (cn) {};
	\draw [connector] (cn) -- (bn) {};	
	
	\draw [connector] (h1) -- (c1) {};
	\draw [connector] (h2) -- (c2) {};
	\draw [connector] (hn) -- (cn) {};
	
	\draw [line] (h1) -- (h2) {};
	\draw [line] (h2) -- (h2r) {};
	\draw [line] (hnl) -- (hn) {};

\end{tikzpicture}
\end{center}
\caption{}
\label{fig:C.c}
\end{subfigure}

\caption{\eqref{fig:C.a} shows the decentralized architecture of the voltage droop controller \eqref{eq:droop}. \eqref{fig:C.b} shows the distributed architecture of Controllers~\eqref{eq:distributed_voltage_control} and \eqref{eq:distributed_voltage_control_third}. \eqref{fig:C.c} shows the architecture of Controller~\eqref{eq:distributed_voltage_control_complete}, with  all-to-all communication.}
\label{fig:control_architecture}
\end{center}
\end{figure*}

\subsection{Distributed averaging controller I}
\label{sec:dist_control_1}
In this section we propose the following distributed controller for MTDC networks which allows for communication between the buses:
\begin{align}
u_i = \; &K_i^P(\hat{V}_i -V_i) \IEEEnonumber \\
\dot{\hat{V}}_i = \; &K^V_i(V^{\text{nom}}_i{-}V_i) \IEEEnonumber \\
&-\gamma \sum_{j\in \mathcal{N}_i} c_{ij} \left( (\hat{V}_i -V_i){-}(\hat{V}_j -V_j) \right),
\label{eq:distributed_voltage_control}
\tag{I}
\end{align}
where $\gamma>0$ is a constant, $K_i^P>0, i=1, \dots, n$, and
\begin{align*}
K^V_i &= \left\{ \begin{array}{ll}
K^V_1>0 & \text{if } i=1 \\
0 & \text{otherwise.}
\end{array} \right.
\end{align*}
This controller can be understood as a proportional control loop (consisting of the first line), and an integral control loop (consisting of the second line). The internal controller variables $\hat{V}_i$ can be understood as reference values for the proportional control loops, regulated by the integral control loop. 
Bus $i=1$, without loss of generality, acts as an integral voltage regulator. The first line of \eqref{eq:distributed_voltage_control} ensures that the controlled injected currents are quickly adjusted after a change in the voltage. The parameter $c_{ij}=c_{ji}>0$ is a constant, and $\mathcal{N}_i$ denotes the set of buses which can communicate with bus $i$. The communication graph is assumed to be undirected, i.e., $j\in \mathcal{N}_i \Leftrightarrow i\in \mathcal{N}_j$. The second line ensures that the voltage is restored at bus $1$ by integral action, and that the controlled injected currents converge to the optimal value, as proven later on. In vector-form, \eqref{eq:distributed_voltage_control} can be written as
\begin{align}
u &= K^P (\hat{V} -V)\nonumber \\
\dot{\hat{V}} &= K^V (V^{\text{nom}}_1 1_{n} -V) -\gamma \mathcal{L}_c (\hat{V}-V) ,\label{eq:distributed_voltage_control_vector}
\end{align}
where $K^P$ is defined as before, $K^V=\diag([K^V_1, 0, \dots, 0])$, and $\mathcal{L}_c$ is the weighted Laplacian matrix of the graph representing the communication topology, denoted $\mathcal{G}_c$, whose edge-weights are given by $c_{ij}$, and which is assumed to be connected.
 The following theorem shows that the proposed controller \eqref{eq:distributed_voltage_control} has desirable properties which the droop controller \eqref{eq:droop} is lacking. It also gives sufficient conditions for which controller parameters stabilize the equilibrium of the closed-loop system.

\begin{theorem}
\label{th:distributed_voltage_control}
Consider an MTDC network described by \eqref{eq:hvdc_scalar}, where the control input $u_i$ is given by \eqref{eq:distributed_voltage_control} and the injected currents $I^{\text{inj}}$ are constant.
The equilibrium of the closed-loop system is stable if
\begin{align}
&\frac{1}{2}\lambda_{\min} \left( (K^P)^{-1}\mathcal{L}_R + \mathcal{L}_R(K^P)^{-1} \right) + 1 \nonumber \\
&+ \frac{\gamma}{2}\lambda_{\min} \left( \mathcal{L}_c (K^P)^{-1} C + C  (K^P)^{-1} \mathcal{L}_c \right)  >0
 \label{eq:stability_cond_distributed_1} \\
& \lambda_{\min} \left(  \mathcal{L}_c (K^P)^{-1} \mathcal{L}_R + \mathcal{L}_R (K^P)^{-1} \mathcal{L}_c \right) \ge 0. \label{eq:stability_cond_distributed_2}
\end{align}
Furthermore, $\lim_{t\rightarrow \infty} V_1(t) = V^{\text{nom}}$, and if $K^P=F^{-1}$ then
$\lim_{t\rightarrow  \infty } u(t) = u^*$. This implies that Objective~\ref{obj:1} is satisfied given that $g_1=1$ and $g_i=0$ for all $i\ge 2$. 
\end{theorem}

\begin{proof}
The closed-loop dynamics of \eqref{eq:hvdc_vector} with the controlled injected currents $u$ given by \eqref{eq:distributed_voltage_control_vector} are given by
\begin{align}
\begin{bmatrix}
\dot{\hat{V}} \\ \dot{V}
\end{bmatrix}
&{=}
\underbrace{
\begin{bmatrix}
-\gamma \mathcal{L}_c & \gamma \mathcal{L}_c -K^V \\
EK^P & -E(\mathcal{L}_R +K^P)
\end{bmatrix}}_{\triangleq A}
\begin{bmatrix}
{\hat{V}} \\ {V}
\end{bmatrix}
{+}
\begin{bmatrix}
K^V V^{\text{nom}}1_{n}  \\
CI^{\text{inj}}
\end{bmatrix}.
\label{eq:closed_loop_distributed_vector}
\end{align}
The characteristic equation of $A$ is given by
\begin{IEEEeqnarray*}{rCl}
0 &=& \det(sI_{2n} -A) = \left| \begin{matrix}
sI_n +\gamma \mathcal{L}_c & -\gamma \mathcal{L}_c + K^V \\
-EK^P & sI_n + E(\mathcal{L}_R +K^P)
\end{matrix} \right| \\
&=& \frac{|CK^P|}{|sI_n+\gamma \mathcal{L}_c|} 
\left| \begin{matrix}
sI_n +\gamma \mathcal{L}_c & -\gamma \mathcal{L}_c + K^V \\
-sI_n-\gamma\mathcal{L}_c & \begin{matrix}
(sI_n+\gamma\mathcal{L}_c)(K^P)^{-1}C\cdot \\ (sI_n + E(\mathcal{L}_R +K^P))
\end{matrix} 
\end{matrix} \right| \\
&=& |CK^P| |(sI_n+\gamma\mathcal{L}_c)(K^P)^{-1}C (sI_n + E(\mathcal{L}_R +K^P)) \\ 
&&-\gamma \mathcal{L}_c + K^V |  \\
&=& |EK^P| \left| (\gamma \mathcal{L}_c(K^P)^{-1}\mathcal{L}_R +K^V) + s((K^P)^{-1}\mathcal{L}_R +I_n \right. \\
&&+\left. \gamma\mathcal{L}_c(K^P)^{-1}C) +s^2((K^P)^{-1}C)  \right|
\\
 &\triangleq & |EK^P| \det(Q(s)).
\end{IEEEeqnarray*}
This assumes that $|sI_n+\gamma \mathcal{L}_c|\ne 0$, however $|sI_n+\gamma \mathcal{L}_c|= 0$ implies $s =0$ or $s\in \mathbb{C}^-$. By elementary column operations, $A$ is shown to be full rank. This still implies that all solutions satisfy $s\in \mathbb{C}^-$. 
Now, the above equation has a solution only if $x^TQ(s)x=0$ for some $x:\; \norm{x}_2=1$. This condition gives the following equation
\begin{IEEEeqnarray*}{rCl}
0&=&  \underbrace{x^T(\gamma \mathcal{L}_c (K^P)^{-1} \mathcal{L}_R + K^V)x}_{a_0}  \\
&& + s\underbrace{x^T((K^P)^{-1}\mathcal{L}_R + I_n + \gamma \mathcal{L}_c(K^P)^{-1}C )x}_{a_1} \\ 
&& + s^2 \underbrace{x^T((K^P)^{-1}C)x}_{a_2},
\end{IEEEeqnarray*}
which by the Routh-Hurwitz stability criterion has all solutions $s\in \mathbb{C}^-$ if and only if $a_i>0$ for $i=0,1,2$. 

Clearly, $a_2>0$, since $((K^P)^{-1}C)$ is diagonal with positive elements. It is easily verified that $a_1>0$ if
\begin{align*}
&\frac{1}{2}\lambda_{\min} \left( (K^P)^{-1}\mathcal{L}_R + \mathcal{L}_R(K^P)^{-1} \right) \\
&+ \frac{\gamma}{2}\lambda_{\min} \left( \mathcal{L}_c (K^P)^{-1} C + C (K^P)^{-1} \mathcal{L}_c \right) + 1 >0.
\end{align*}
Finally, clearly $x^T( \mathcal{L}_c (K^P)^{-1} \mathcal{L}_R)x \ge 0$ for any $x:\; \norm{x}_2=1$ if and only if
\begin{align*}
\frac 12 \lambda_{\min} \left(  \mathcal{L}_c (K^P)^{-1} \mathcal{L}_R + \mathcal{L}_R (K^P)^{-1} \mathcal{L}_c \right) &\ge 0.
\end{align*}
Since the graphs corresponding to $\mathcal{L}_R$ and $\mathcal{L}_c$ are both assumed to be connected, the only $x$ for which $x^T( \mathcal{L}_c (K^P)^{-1} \mathcal{L}_R)x = 0$ is $x=\frac{1}{\sqrt{n}}[1, \dots, 1]^T$. Given this $x=\frac{1}{\sqrt{n}}[1, \dots, 1]^T$, $x^T K^Vx=\frac{1}{n} K^V_1>0$. Thus, $a_0>0$ gives that the above inequality holds. Thus, under assumptions \eqref{eq:stability_cond_distributed_1}--\eqref{eq:stability_cond_distributed_2}, $A$ is Hurwitz, and thus the equilibrium of the closed-loop system is stable. 

Now consider the equilibrium of \eqref{eq:closed_loop_distributed_vector}. Premultiplying the first $n$ rows with $1_n^T$ yields $0=1_n^T K^V(V^{\text{nom}}1_{n}-V) = K^V_1(V^{\text{nom}}-V_1)$. Clearly this minimizes \eqref{eq:cost_voltage}, with the minimal value $0$. Inserting this back to the first $n$ rows of \eqref{eq:closed_loop_distributed_vector} yields $0=\mathcal{L}_c (V-\hat{V})$, implying that $(V-\hat{V}) = k1_{n}$. It should be noted here that if $K^V_i > 0$ for at least one $i \ge 2$, then the first $n$ rows of \eqref{eq:closed_loop_distributed_vector} do not imply $(V-\hat{V}) = k1_{n}$ in general. 
 Inserting the relation $(V-\hat{V}) = k1_{n}$ in \eqref{eq:distributed_voltage_control_vector} gives $u=K^P(V-\hat{V}) =  k K^P1_{n}$. Setting $K^P=F^{-1}$, \eqref{eq:opt_u}--\eqref{eq:opt1} are satisfied by Lemma \ref{lemma:optimality}. 
\end{proof}

\begin{remark}
\label{rem:1}
For sufficiently uniformly large $K^P$, and sufficiently small $\gamma$, the condition \eqref{eq:stability_cond_distributed_1} is fulfilled. However, stability is independent of $K^V$. 
\end{remark}

\begin{corollary}
\label{cor:1}
A sufficient condition for when \eqref{eq:stability_cond_distributed_2} is fulfilled, is that $\mathcal{L}_c=\mathcal{L}_R$,  i.e., the topology of the communication network is identical to the topology of the power transmission lines and the edge weights of the graphs are identical.
\end{corollary}

\subsection{Distributed averaging controller II}
\label{sec:dist_control_2}
While the controller \eqref{eq:distributed_voltage_control} is clearly distributed, it has poor redundancy due to a specific HVDC bus dedicated for voltage measurement. Should the dedicated bus fail, the voltage of bus $1$ will not converge to the reference voltage asymptotically. To improve the redundancy of \eqref{eq:distributed_voltage_control}, we propose the following controller:
\begin{align}
u_i = \; & K_i^P(\hat{V}_i -V_i) \nonumber \\
\dot{\hat{V}}_i = \; & k^V \sum_{i\in\mathcal{V}} (V^{\text{nom}}_i - V_i) \nonumber \\
&-\gamma \sum_{j\in \mathcal{N}_i} c_{ij} \left( (\hat{V}_i -V_i){-}(\hat{V}_j -V_j) \right),\tag{II}
\label{eq:distributed_voltage_control_complete}
\end{align}
where $\gamma>0$ and $k^V>0$ are constants.
This controller can as \eqref{eq:distributed_voltage_control} be interpreted as a fast proportional control loop (consisting of the first line), and a slower integral control loop (consisting of the second and third lines). In contrast to \eqref{eq:distributed_voltage_control} however, every bus implementing \eqref{eq:distributed_voltage_control_complete} requires voltage measurements from all buses of the MTDC system. Thus, controller~\eqref{eq:distributed_voltage_control_complete} requires a complete communication graph. As long as the internal controller dynamics of $\hat{V}$ are sufficiently slow (e.g., by choosing $k^V$ sufficiently small), this is a reasonable assumption provided that a connected communication network exists.  In vector-form, \eqref{eq:distributed_voltage_control_complete} can be written as
\begin{align}
\label{eq:distributed_voltage_control_complete_vector}
\begin{aligned}
u &= K^P (\hat{V} -V) \\
\dot{\hat{V}} &= k^V 1_{n\times n} (V^{\text{nom}} -V) -\gamma \mathcal{L}_c (\hat{V}-V) ,
\end{aligned}
\end{align}
where $K^P$ is defined as before, $V^{\text{nom}} = [V^{\text{nom}}_1, \dots, V^{\text{nom}}_n]^T$ and $\mathcal{L}_c$ is the weighted Laplacian matrix of the graph representing the communication topology, denoted $\mathcal{G}_c$, whose edge-weights are given by $c_{ij}$, and which is assumed to be connected.
 The following theorem is analogous to Theorem~\ref{th:distributed_voltage_control}, and gives sufficient conditions for which controller parameters result in a stable equilibrium of the closed-loop system.

\begin{theorem}
\label{th:distributed_voltage_control_complete}
Consider an MTDC network described by \eqref{eq:hvdc_scalar}, where the control input $u_i$ is given by \eqref{eq:distributed_voltage_control_complete} and the injected currents $I^{\text{inj}}$ are constant.
The equilibrium of the closed-loop system is stable if \eqref{eq:stability_cond_distributed_1} and \eqref{eq:stability_cond_distributed_2} are satisfied. 
If furthermore $K^P=F^{-1}$, then $\lim_{t\rightarrow  \infty } u(t) = u^*$, and if $G=I_n$, where $G=\diag([g_1, \dots, g_n])$, \eqref{eq:cost_voltage} is minimized.  
This implies that Objective \ref{obj:1} is satisfied.
\end{theorem}

\begin{proof}
The proof is analogous to the proof of Theorem~\ref{th:distributed_voltage_control}. Since $x^T1_{n\times n}x = (1_n^Tx)^T(1_n^Tx)\ge 0$, $1_{n\times n}\ge 0$, implying that the term $a_0$ is positive if
\begin{align*}
\frac 12 \lambda_{\min} \left(  \mathcal{L}_c (K^P)^{-1} \mathcal{L}_R + \mathcal{L}_R (K^P)^{-1} \mathcal{L}_c \right) &\ge 0.
\end{align*}
Thus the matrix $A$ is Hurwitz whenever \eqref{eq:stability_cond_distributed_1} and \eqref{eq:stability_cond_distributed_2} are satisfied. 
The equilibrium of the closed-loop system implies that $1_n^T(V^\text{nom} - V) = 0$. Thus, by Lemma~\ref{cor:optimality_voltage}, Equation~\eqref{eq:cost_voltage} is minimized. 
The remainder of the proof is identical to the proof of Theorem~\ref{th:distributed_voltage_control}, and is omitted. 
\end{proof}
\begin{remark}
\label{rem:controller_II}
For sufficiently uniformly large $K^P$, and sufficiently small $\gamma$, the condition \eqref{eq:stability_cond_distributed_2} is fulfilled. However, stability is independent of $K^V$. 
\end{remark}

\subsection{Distributed averaging controller III}
\label{sec:dist_control_3}
While the assumption that the voltage measurements can be communicated instantaneously through the whole MTDC network is reasonable for small networks or slow internal controller dynamics, the assumption might be unreasonable for larger networks. To overcome this potential issue, a novel controller is proposed. 
 The proposed controller takes inspiration from the control algorithms given in \cite{andreasson2013distributed, Andreasson2014_IFAC, simpson2012synchronization}, and is given by
\begin{align}
u_i &= -K^P_i(V_i - \hat{V}_i - \bar{V}_i ) \nonumber \\
\dot{\hat{V}}_i &= -\gamma \sum_{j\in \mathcal{N}_i} c_{ij} \left( (\hat{V}_i + \bar{V}_i -V_i){-}(\hat{V}_j + \bar{V}_j -V_j) \right) \nonumber \\
\dot{\bar{V}}_i &= - K_i^V (V_i - V_i^{\text{nom}}) -\delta \sum_{j\in \mathcal{N}_i} c_{ij} (\bar{V}_i - \bar{V}_j).\tag{III}
\label{eq:distributed_voltage_control_third}
\end{align}
The first line of the controller \eqref{eq:distributed_voltage_control_third} can be interpreted as a proportional controller, whose reference value is controlled by the remaining two lines. The second line ensures that the weighted current injections converge to the identical optimal value through a consensus-filter. The third line is a distributed secondary voltage controller, where each bus measures the voltage and updates the reference value through a consensus-filter. 
 In vector form, \eqref{eq:distributed_voltage_control_third} can be written as
\begin{align}
u &= -K^P (V - \hat{V} - \bar{V} ) \nonumber \\
\dot{\hat{V}} &= -\gamma \mathcal{L}_c (\hat{V} + \bar{V} -V) \nonumber \\
\dot{\bar{V}} &= -K^V(V - V^{\text{nom}}) - \delta \mathcal{L}_c \bar{V},\label{eq:distributed_voltage_control_third_vector}
\end{align}
where $K^P=\diag([K^P_1, \dots , K^P_n])$, $K^V=\diag([K^V_1, \dots, K^V_n])$, $V^{\text{nom}}=[V^{\text{nom}}_1, \dots , V^{\text{nom}}_n]^T$ and $\mathcal{L}_c$ is the weighted Laplacian matrix of the graph representing the communication topology, denoted $\mathcal{G}_c$, whose edge-weights are given by $c_{ij}$, and which is assumed to be connected.
Substituting the controller \eqref{eq:distributed_voltage_control_third_vector} in the system dynamics \eqref{eq:hvdc_vector}, yields
\begin{IEEEeqnarray}{rCl}
\begin{bmatrix}
\dot{\bar{V}} \\ \dot{\hat{V}} \\ \dot{V}
\end{bmatrix}
&=&
\underbrace{
\begin{bmatrix}
-\delta \mathcal{L}_c & 0_{n\times n} & -K^V \\
-\gamma \mathcal{L}_c & -\gamma \mathcal{L}_c & \gamma \mathcal{L}_c \\
EK^P & EK^P & -E(\mathcal{L}_R + K^P)
\end{bmatrix}}_{\triangleq A}
\begin{bmatrix}
{\bar{V}} \\ {\hat{V}} \\ {V}
\end{bmatrix} \IEEEnonumber 
\\
&&+ \underbrace{\begin{bmatrix}
K^VV^{\text{nom}} \\ 0_{n} \\ EI^{\text{inj}}
\end{bmatrix}}_{\triangleq b}.
\label{eq:cl_dynamics_vector_third}
\end{IEEEeqnarray}
 The following theorem characterizes when the controller \eqref{eq:distributed_voltage_control} stabilizes the equilibrium of \eqref{eq:hvdc_scalar}, and shows that it has some desirable properties. 

\begin{theorem}
\label{th:distributed_voltage_control_third}
Consider an MTDC network described by \eqref{eq:hvdc_scalar}, where the control input $u_i$ is given by \eqref{eq:distributed_voltage_control_third} and the injected currents $I^{\text{inj}}$ are constant. If all eigenvalues of $A$, except the one eigenvalue which always equals $0$, lie in $\mathbb{C}^-$, $K^P=F^{-1}$ and $K^V=G$, where $G=\diag([g_1, \dots, g_n])$, then Objective \ref{obj:1} is satisfied given any non-negative constants $g_i, \; i=1,\dots, n$. 
\end{theorem}

\begin{proof}
It is easily shown that  $A$ as defined in \eqref{eq:cl_dynamics_vector_third}, has one eigenvalue equal to $0$. 
The right-eigenvector of $A$ corresponding to the zero eigenvalue is $v_1={1}/{\sqrt{2n}}[1_{n}^T, -1_{n}^T, 0_{n}^T]^T$. Since $b$ as defined in \eqref{eq:cl_dynamics_vector_third}, is not parallel to $v_1$, $\lim_{t\rightarrow \infty} [{\bar{V}}(t), {\hat{V}}(t), {V}(t)]$ exists and is finite, by the assumption that all other eigenvalues lie in $\mathbb{C}^-$. Hence, we consider any stationary solution of \eqref{eq:cl_dynamics_vector_third}
\begin{IEEEeqnarray}{rCl}
{
\begin{bmatrix}
\delta \mathcal{L}_c & 0_{n\times n} & K^V \\
\gamma \mathcal{L}_c & \gamma \mathcal{L}_c & -\gamma \mathcal{L}_c \\
-K^P & -K^P & (\mathcal{L}_R + K^P)
\end{bmatrix}
\begin{bmatrix}
{\bar{V}} \\ {\hat{V}} \\ {V}
\end{bmatrix}}
{=} {\begin{bmatrix}
K^VV^{\text{nom}} \\ 0_{n} \\ I^{\text{inj}}
\end{bmatrix}}.
\label{eq:cl_dynamics_vector_equilibrium_third}
\end{IEEEeqnarray}
Premultiplying \eqref{eq:cl_dynamics_vector_equilibrium_third} with $[1_{n}^T, 0_{n}^T, 0_{n}^T]$ yields
\begin{align*}
1_{n}^TK^V ( V^{\text{nom}}-V) = \sum_{i=1}^n K_i^V ( V_i^{\text{nom}} - V_i^{\text{nom}}) =0,
\end{align*}
which by Lemma~\ref{cor:optimality_voltage} implies that Equation~\eqref{eq:cost_voltage} is minimized. 
The $n+1$-th to $2n$-th lines of \eqref{eq:cl_dynamics_vector_equilibrium_third} imply
$\mathcal{L}_c(\bar{V}+\hat{V}-V) = 0_{n}  \Rightarrow  (\bar{V}+\hat{V}-V) = k_1 1_{n}   
\Rightarrow 
u = K^P (\bar{V}+\hat{V}-V) =  k_1 K^P 1_{n}
$
By Lemma \ref{lemma:optimality}, \eqref{eq:opt_u}--\eqref{eq:opt1} are satisfied. 
\end{proof}
Note that Controller~\eqref{eq:distributed_voltage_control_third} is the only controller among the presented controllers which minimizes Equation~\eqref{eq:cost_voltage}, for any a priori given constants $g_i, \; i=1, \dots, n$. Controllers~\eqref{eq:distributed_voltage_control} and \eqref{eq:distributed_voltage_control_complete} minimize Equation~\eqref{eq:cost_voltage}, but for specific values of $g_i, \; i=1, \dots, n$. 
While Theorem~\ref{th:distributed_voltage_control_third} establishes an exact condition when the distributed controller \eqref{eq:distributed_voltage_control_third} stabilizes the equilibrium of the MTDC system \eqref{eq:hvdc_scalar}, it does not give any insight in how to choose the controller parameters to stabilize the equilibrium. The following theorem gives a sufficient stability condition for a special case.
\begin{theorem}
\label{th:stability_A}
Assume that $\mathcal{L}_c = \mathcal{L}_R$, i.e. that the topology of the communication network is identical to the topology of the MTDC system. Assume furthermore that $K^P = k^P I_n$, i.e. the controller gains are equal. 
 Then all eigenvalues of $A$ except the zero eigenvalue lie in $\mathbb{C}^-$ if
 \begin{IEEEeqnarray}{Cl}
 &\frac{\gamma+\delta}{2k^P} \lambda_{\min}\left( \mathcal{L}_RC +C\mathcal{L}_R \right) + 1 >0 \label{eq:stability_A_sufficient_1} \\
  &\frac{\gamma\delta}{2k^P} \lambda_{\min}\left( \mathcal{L}^2_RC + C\mathcal{L}_R^2 \right) + \min_i K^V_i >0 \label{eq:stability_A_sufficient_2} \\ 
  &\lambda_{\max}\left(\mathcal{L}_R^3\right) \frac{\gamma\delta}{k^{P^2}} \IEEEnonumber \\
  \le & \left( \frac{\gamma+\delta}{2k^P} \lambda_{\min}\left( \mathcal{L}_RC  +C\mathcal{L}_R \right) + 1 \right) \IEEEnonumber \\
   &\times \left( \frac{\gamma\delta}{2k^P} \lambda_{\min}\left( \mathcal{L}^2_RC + C\mathcal{L}_R^2 \right) + \min_i K^V_i \right)
    \label{eq:stability_A_sufficient_3}
 \end{IEEEeqnarray}
\end{theorem}

\begin{proof}
Following similar steps as the proof of Theorem~\ref{th:distributed_voltage_control}, one obtains after some tedious matrix manipulations that \eqref{eq:cl_dynamics_vector_third} is stable if the following equation has solutions $s \in \mathbb{C}^-$:
\begin{IEEEeqnarray}{ll}
&0= x^TQ(s)x =  \underbrace{ \frac{\gamma\delta}{k^P} x^T \mathcal{L}_R^3 x}_{a_0} \IEEEnonumber \\
&+ s \underbrace{x^T\left[ \frac{\delta+\gamma}{k^P}\mathcal{L}_R^2 
+ \delta\mathcal{L}_R +\frac{\gamma\delta}{k^P}\mathcal{L}_R^2C +K^V \right]x}_{a_1} \IEEEnonumber \\
&+   s^2 \underbrace{x^T\left[\frac{1}{k^P}\mathcal{L}_R {+} I_n {+}\frac{\gamma+\delta}{k^P}\mathcal{L}_RC \right]x}_{a_2}
{+} s^3 \underbrace{\frac{1}{k^P} x^T C x}_{a_3}. \quad
\label{eq:x^TQx}
\end{IEEEeqnarray}
Clearly \eqref{eq:x^TQx} has one solution $s=0$ for $x=\frac{a}{\sqrt{n}}[1, \dots, 1]^T$, since this implies that $a_0=0$. The remaining solutions are stable if and only if the polynomial $a_1+ sa_2+s^2a_3=0$ is Hurwitz, which is equivalent to $a_i>0$ for $i=1,2,3$ by the Routh-Hurwitz stability criterion. For $x\ne \frac{a}{\sqrt{n}}[1, \dots, 1]^T$, we have that $a_0>0$, and thus $s=0$ cannot be a solution of \eqref{eq:x^TQx}. By the Routh-Hurwitz stability criterion, \eqref{eq:x^TQx} has stable solutions if and only if $a_i>0$ for $i=0,1,2,3$ and $a_0a_3<a_1a_2$. Since this condition implies that $a_i>0$ for $i=1,2,3$, there is no need to check this second condition explicitly. Clearly $a_3>0$ since $(K^P)^{-1}$ and $C$ are diagonal with positive elements. It is easily verified that $a_2>0$ if \eqref{eq:stability_A_sufficient_1} holds, since $\mathcal{L}_R\ge 0$. Similarly, $a_1>0$ if \eqref{eq:stability_A_sufficient_2} holds, since also $\mathcal{L}_R^2\ge0$ and $x^TK^Vx\ge \min_i K^V_i$. In order to assure that $a_0a_3<a_1a_2$, we need furthermore to upper bound $a_0a_3$. The following bound is easily verified
\begin{align*}
a_0a_3<\lambda_{\max}\left(\mathcal{L}_R^3\right) \frac{\gamma\delta}{k^{P^2}} \max_i C_i.
\end{align*} 
Using this, together with the lower bounds on $a_1$ and $a_2$, we obtain that \eqref{eq:stability_A_sufficient_3} is a sufficient condition for $a_0a_3<a_1a_2$. 
\end{proof}

\begin{remark}
For sufficiently small $\gamma$ and $\delta$, and sufficiently large  $k^P$ and $\min_iK^V_i$, the inequalities \eqref{eq:stability_A_sufficient_1}--\eqref{eq:stability_A_sufficient_3} hold, thus always enabling the choice of stabilizing controller gains. 
\end{remark}

\section{Simulations}
\label{sec:simulations}
Simulations of an MTDC system were conducted using MATLAB. The MTDC  was modelled by \eqref{eq:hvdc_scalar}, with $u_i$ given by the distributed controllers \eqref{eq:distributed_voltage_control}, \eqref{eq:distributed_voltage_control_complete} and \eqref{eq:distributed_voltage_control_third}, respectively. The topology of the MTDC system is assumed to be as illustrated in Figure~\ref{fig:graph}. The system parameter values are obtained from \cite{jovcic2003vsc}, where the inductances of the DC lines are neglected, and the capacitances of the DC lines are assumed to be located at the converters. The system parameter values are assumed to be identical for all converters, and are given in Table~\ref{table:system_parameters}. The controller parameters are also assumed to be uniform, i.e., $K^P_i= k_p, K^V_i=k^V$ for $i=1,2,3,4$, and their numerical values are given in Table \ref{table:controller_parameters}. Due to the communication of controller variables, a constant delay of $500$ ms is assumed. The delay only affects remote information, so that, e.g., the first line of the controllers \eqref{eq:distributed_voltage_control}, \eqref{eq:distributed_voltage_control_complete} and \eqref{eq:distributed_voltage_control_third} remain delay-free. 
\begin{table}
\caption{System parameter values used in the simulation.}
\label{table:system_parameters}
\centering
\begin{tabular}{l|l}
Parameter & Value [Unit] \\
\midrule
$C_i$ & $57 \mu$F \\
$R_{ij}$ & $3.7 \Omega$
\end{tabular}
\end{table}

\begin{table}
\caption{Controller parameter values used in the simulation.}
\label{table:controller_parameters}
\centering
\begin{tabular}{l|lll}
Parameter/Controller & I & II & III \\
\midrule
$k^P$ & 10 $\Omega^{-1}$ & 10 $\Omega^{-1}$ & 0.5 $\Omega^{-1}$ \\

$k^V$  & 10 & 5 & 2.5 \\

$\gamma$  & 20 & 15 & 3 \\

$\delta$  & - & - & 2 \\
\end{tabular}
\end{table}
The communication gains were set to $c_{ij} = R_{ij}^{-1} \;$S for all $(i,j)\in \mathcal{E}$ and for all controllers. 
%
The injected currents are assumed to be initially given by $I^{\text{inj}}=[300,200,-100,-400]^T$ A, and the system is allowed to converge to its equilibrium. Since the injected currents satisfy $I_i^{\text{inj}}=0$, $u_i=0$ for $i=1,2,3,4$ by Theorem~\ref{th:distributed_voltage_control}.
Then, at time $t=0$, the injected currents are changed to $I^{\text{inj}}=[  300, 200, -300, -400]^T$ A. The step responses of the voltages $V_i$ and the controlled injected currents $u_i$ are shown in Figure~\ref{fig:powersystems_sim_1_tau_0}. The conservative voltage bounds implied by Lemma \ref{lemma:voltage_differences}, are indicated by the two dashed lines. We note that the controlled injected currents $u_i$ converge to their optimal values, and that the voltages remain within the bounds. 

\setlength\fheight{2.1cm}
\setlength\fwidth{6.8cm}

\pgfplotsset{yticklabel style={text width=3em,align=right}}
\begin{figure*}
 \centering
 		\renewcommand{\thesubfigure}{I.a}
        \begin{subfigure}[b]{0.47\textwidth}
        \begin{flushleft}
        			\hspace{-0.25cm}
                \input{Simulations_leader/V_tau=0gamma=0005.tikz}
                \caption{Controller I, bus voltages} 
                \label{fig:gull}
        \end{flushleft}
        \end{subfigure} \quad 
        ~ 
          \renewcommand{\thesubfigure}{I.b}
        \begin{subfigure}[b]{0.47\textwidth}
        \begin{flushright}
                \input{Simulations_leader/u_tau=0gamma=0005.tikz}
                \caption{Controller I, controlled injected currents}
                \label{fig:tiger}    
        \end{flushright}
        \end{subfigure}
        ~ 
      	\renewcommand{\thesubfigure}{II.a}
        \begin{subfigure}[b]{0.47\textwidth}
        \begin{flushleft}
        			\hspace{-0.25cm}
                \input{Simulations_complete/V_tau=0gamma=0005.tikz}
                \caption{Controller II, bus voltages}
                \label{fig:gull}        
        \end{flushleft}
        \end{subfigure}\quad
        ~ 
          \renewcommand{\thesubfigure}{II.b}
        \begin{subfigure}[b]{0.47\textwidth}
        \begin{flushright}
                \input{Simulations_complete/u_tau=0gamma=0005.tikz}
                \caption{Controller II, controlled injected currents}
                \label{fig:tiger}    
        \end{flushright}
        \end{subfigure}
        
        \renewcommand{\thesubfigure}{III.a}
        \begin{subfigure}[b]{0.47\textwidth}
        \begin{flushleft}
        			\hspace{-0.25cm}
                \input{Simulations_leaderless/V_tau=0gamma=0005.tikz}
                \caption{Controller III, bus voltages}
                \label{fig:gull}        
        \end{flushleft}
        \end{subfigure}\quad
        ~ 
          \renewcommand{\thesubfigure}{III.b}
        \begin{subfigure}[b]{0.47\textwidth}
        \begin{flushright}
                \input{Simulations_leaderless/u_tau=0gamma=0005.tikz}
                \caption{Controller III, controlled injected currents}
                \label{fig:tiger}     
        \end{flushright}
        \end{subfigure}
        \caption{The figures show the voltages $V_i$ and the controlled injected currents $u_i$, respectively. The system model is given by \eqref{eq:hvdc_scalar}, and $u_i$ is given by the distributed controllers \eqref{eq:distributed_voltage_control}, \eqref{eq:distributed_voltage_control_complete} and \eqref{eq:distributed_voltage_control_third}, respectively. We note that all controllers demonstrate reasonable performance. Controller~\eqref{eq:distributed_voltage_control_third} has the advantage of being fully distributed, while Controller~\eqref{eq:distributed_voltage_control} requires a dedicated voltage measurement bus, and Controller~\eqref{eq:distributed_voltage_control_complete} requires a complete communication network. }
        \label{fig:powersystems_sim_1_tau_0}
\end{figure*}

\section{Discussion and Conclusions}
\label{sec:discussion}
In this paper we have studied VDM for MTDC systems, and highlighted some of its weaknesses. To overcome some of its disadvantages, three distributed controllers for MTDC systems were proposed. We showed that under certain conditions, there exist controller parameters such that the equilibria of the closed-loop systems are stabilized. 
In particular, a sufficient stability condition is that the graphs of the physical MTDC network and the communication network are identical, including their edge weights. 
We have shown that the proposed controllers are  able to control the voltage levels of the DC buses close towards the nominal voltages, while simultaneously minimizing a quadratic cost function of the current injections. The proposed controllers were tested on a four-bus MTDC network by simulation, demonstrating their effectiveness. 

This paper lays the foundation for distributed control strategies for hybrid AC and MTDC systems. Future work will in addition to the voltage dynamics of the MTDC system, also consider the dynamics of connected AC systems. Interconnecting multiple asynchronous AC systems also enables novel control applications, for example automatic sharing of primary and secondary frequency control reserves. Preliminary results on decentralized cooperative AC frequency control by an MTDC grid have been presented in \cite{andreasson2014mtdacac}.
The stability conditions in this work depend on both products of diagonal matrices and Laplacian matrices, and products of Laplacian matrices. Little is known about the positive definiteness of such matrix products, motivating further research. 

\bibliography{references}
\bibliographystyle{plain}
\end{document}